\documentclass{article}
\usepackage{lineno,hyperref}
\usepackage{epsfig}
\usepackage{graphicx}
\usepackage{epstopdf}
\usepackage{amssymb,amsmath,amsthm}
\newtheorem{proposition}{Proposition}
\newtheorem{corollary}{Corollary}

\begin{document}

\title{Explicit formulae for derivatives and primitives of orthogonal polynomials\thanks{Research funded by the European Regional Development Fund through the program COMPETE and by the Portuguese Government through the FCT -- Funda\c c\~ao para a Ci\^encia e a Tecnologia under the project PEst-C/MAT/UI0144/2013.}}

\author{Jos\'e M. A. Matos\footnote{Centro de Matem\'atica da Universidade do Porto and Instituto Superior de Engenharia do Porto, Rua Dr. Ant\'onio Bernardino de Almeida, 431, 4200-072 Porto, Portugal, jma@isep.ipp.pt}, Maria Jo\~ao Rodrigues\footnote{Faculdade de Ci\^encias da Universidade do Porto and Centro de Matem\'atica da Universidade do Porto}, Jo\~ao Carrilho de Matos\footnote{Instituto Superior de Engenharia do Porto}}

\maketitle

\begin{abstract}
In this work we deduce explicit formulae for the elements of the matrices that represent the action of integro-differential operators over the coefficients of generalized Fourier series.
Our formulae are obtained by performing operations on the bases of orthogonal polynomials and result directly from the three-term recurrence relation satisfied by the polynomials.
Moreover we give exact formulae for the coefficients for some families of orthogonal polynomials. 

Some tests are given to demonstrate the robustness of the formulas presented.

\end{abstract}

\textbf{keywords:} Integro-differential equations; Operational matrices;  Orthogonal polynomials

\section{Introduction}

In some spectral methods the operational matrices that transform integro-differential problems into algebraic problems are often obtained using a similarity transformation \cite{OS}. For high degree approximation the accuracy of the approximate solutions is degraded by the bad conditioning of the matrices involved. 
In recent works, dealing with the extension of spectral methods to systems of nonlinear integro-differential problems \cite{VMT17} and to problems with non-polynomial coefficients \cite{TVM17}, the error propagation when working with operational matrices is referred as a drawback.  This fact is of great importance when there is need of a large number of coefficients  computed with great precision, as it is in the case with  Frobenius-Pad\'e approximation allowing the computation of rational approximants of series with unknown coefficients \cite{MMR}.

Several authors have studied different approaches  with the purpose of stabilizing the spectral method, by introducing modifications in the way of representing the solution and thus obtaining a better conditioned algebraic system. The idea is to represent the action of integro-differential operators acting on a basis of orthogonal polynomials, on the same basis. But these works refer, in general, to very particular cases either of  operators or of orthogonal bases. For example see, for  Chebyshev polynomials \cite{DSW}, \cite{Greengard}, \cite{KW}, \cite{SKR}, for Legendre polynomials \cite{CW}, \cite{Koepf}, \cite{KD}, \cite{Phillips}, for Jacobi polynomials \cite{Doha}, \cite{KD}, for Bessel polynomials \cite{DA}, \cite{KD} or for Hermite and Laguerre polynomials \cite {KD}.
Here we go in this same direction of avoiding the similarity transformation, but our work is more general and results only and directly from the  three term recurrence relation satisfied by the orthogonal polynomials used. We have established recurrence relations for the operational matrix elements in general,  and   we  give explicit formulas for those elements when we consider families of classical orthogonal polynomials.

Our procedure being more general,  unifies and  includes  the  cases cited above, with advantages from the point of view of  the automation of the algorithms as well as  from   the numerical point of view.

In Section 2  we present  the recursive formulas for the representation of derivative and primitive sequences of orthogonal polynomials expressed in terms of the polynomials themselves. The concretization of these formulas in particular cases of families of classical orthogonal polynomials results in explicit formulas, some of which have already been mentioned in the literature. In Section 3 the results of the previous section are used for the algebraic representation of integral-differential operators and  in section 4 some numerical tests to the robustness of the developed formulas are presented.

\section{Algebraic and Analytic Operations on Polynomials}\label{Oper_on_Polyn}

In this section we introduce a set of formulas representing the effect of integro-differential operations over the  coefficients of a formal series. Those formulas are presented in the form of matrix operations and we will present exact formulas to their elements.

First we present formulas for generic orthogonal polynomials. Some of them are explicit formulas and others are given as recurrence relations. In a second subsection we present explicit formulas to the case of classical orthogonal polynomials.

\subsection{Formulas for general orthogonal polynomials}

In what follows, ${\cal P}=(P_0,P_1,\ldots)$ is the orthogonal polynomial sequence defined by an inner product
\begin{equation}\label{PiPj}
<P_i,P_j>=\int^{b}_{a} P_i\ P_j\ w dx = ||P_i||^2\delta_{ij},\ i,j\in\mathbb{N}_0 
\end{equation}
where $\delta_{ij}$ is the Kronecker symbol.
A well known property of $\cal P$ is that it constitutes a basis for $\mathbb{P}$, the space of polynomials of any degree. Another property of $\cal P$ concerns the coefficients of formal series.

\begin{proposition}
       Lets $f$ be a function represented by an expansion over a basis of orthogonal polynomials $\mathcal{P}=(P_0,P_1,\ldots)$ satisfying  \eqref{PiPj},
       $$
       f=\sum_{i=0}^{\infty}f_{i}P_i
       $$
       \noindent then,
       \begin{equation}\label{fi}
	 f_i=\frac{1}{||P_i||^2}<P_i,f> 
	 \end{equation}
\noindent where, the equality only holds when the infinite series converge to $f$.
\end{proposition}

 This is the key property to show the following proposition.

\begin{proposition}\label{prop:L}
	If $l$ is a linear operator acting on $\mathbb{P}$ and $L$ is the infinite matrix defined by
	\begin{equation}
	L=[l_{ij}]_{i,j\geq 0},\quad \text{with}\quad l_{ij} = \frac{1}{||P_i||^2} <P_i,lP_j>,
	\end{equation}
	then formally $l{\cal P}={\cal P}L$.
	\begin{proof}
		For each $j=0,1,\ldots$ we define the infinite unitary vector $e_j=[\delta_{ij}]_{i\geq 0}$. So that ${\cal P} e_j = P_j$ and using \eqref{fi} we get
		\[ l {\cal P} e_j = l P_j = \sum_{i\geq 0} \frac{<P_i, l P_j>}{||P_i||^2}  P_i = \sum_{i\geq 0} l_{ij} P_i = {\cal P} L e_j,\ j=0,1,\ldots  \]
		and so  $l {\cal P}={\cal P} L$, in the element wise sense.
	\end{proof} 
\end{proposition}

A characteristic property of orthogonal polynomials is that they satisfies a three term recurrence relation. The actual values for the recurrence relation coefficients depends on a normalization choice. In the sequel we always consider $P_0=1$ and the recurrence relation in the form
\begin{equation}\label{ttrr}
\left\{\begin{array}{l}
x P_j = \alpha_j P_{j+1} + \beta_j P_{j} + \gamma_j P_{j-1},\ j\geq 0 \\
P_{-1}=0,\ P_0=1
\end{array}\right. ,
\end{equation}
and we will show that this is enough to determine matrices $L$ representing in $\cal P$ the action of linear integro-differential operators $l$.

We start by the operator that to each element $p\in\mathbb{P}$ associate the polynomial $x p$, usually referred as the shift operator.

\begin{proposition}\label{prop:xP}
	Let $\cal P$ be a basis satisfying \eqref{ttrr}, defining 
	\[
	M=[\mu_{i,j}]_{i,j\geq 0},\quad \text{with}\quad \mu_{i,j} = \frac{1}{||P_i||^2}<P_i,x P_j>,
	\]
	then 
	\begin{equation}\label{mu} 
		\begin{cases}
		\mu_{0,0} = \beta_0,\ \mu_{1,0} =  \alpha_0 \\
		\mu_{j-1,j} = \gamma_j,\  \mu_{j,j} = \beta_j,\ \mu_{j+1,j} =  \alpha_j \\ 
		\mu_{i,j} = 0,\ |i-j|>1 
		\end{cases},\ j=1,2,\ldots	
	\end{equation}
	and  $x {\cal P}={\cal P} M$.
	\begin{proof}
		From the definition of $\mu_{i,j}$ and \eqref{fi} follows that
		\[ xP_j=\sum_{i=0}^{j+1} \mu_{i,j} P_i,\ j=0,1,\ldots \]
		and using \eqref{ttrr} we get \eqref{mu}. The fact that $x {\cal P}={\cal P} M$, in the element wise sense, is a consequence of proposition \ref{prop:L}.
	\end{proof} 
\end{proposition}

For the differential operator, we define a matrix $N = [\eta_{i,j}]_{i,j\geq 0}$ such that $P'_j=\sum_{i=0}^{j-1}\eta_{ij}P_i,\ j\geq 0$. So $\eta_{ij}=0,\ i\geq j$, that is, $N$ is an upper triangular matrix with null main diagonal. Using \eqref{ttrr} leads to a relation between neighbour matrix elements, allowing their evaluation by recurrence. 

\begin{proposition}\label{prop:dP}
	Let $\cal P$ be a basis satisfying \eqref{ttrr}, defining 
	\[
	N = [\eta_{i,j}]_{i,j\geq 0},\quad \text{with}\quad \eta_{i,j} = \frac{1}{||P_i||^2}< P_i, \frac{d}{dx} P_j >,
	\]
	then $\eta_{i,0}=0,\ i\geq 0$, $\eta_{0,1}=1/\alpha_0$ and for $j=1,2,\ldots$
	\begin{equation}\label{eta} 
	 	\begin{cases}
		\eta_{i,j+1}=\frac{1}{\alpha_j}\big[\alpha_{i-1}\eta_{i-1,j} + (\beta_i-\beta_j)\eta_{i,j} \\ \hspace{0.24\textwidth} + \gamma_{i+1}\eta_{i+1,j} - \gamma_j\eta_{i,j-1}\big],\ i=0,\ldots,j-1\\
		\eta_{j,j+1}=\frac{1}{\alpha_j}\alpha_{j-1}\eta_{j-1,j}
		\end{cases},
	\end{equation}
	and  $\frac{d}{dx} {\cal P}={\cal P} N$.
	\begin{proof}
		From \eqref{ttrr} we have $P_0=1$ and 
		$ P_1 = \frac{1}{\alpha_0}(x-\beta_{0}) $
		and so $\eta_{i,0}=0,\ i\geq 0$ and $\eta_{0,1}=\frac{1}{\alpha_0}$. For the remain columns of $N$, applying the operator $\frac{d}{dx}$ to both sides of \eqref{ttrr} then
		\[ P_j + x P'_j = \alpha_j P'_{j+1} + \beta_j P'_{j} + \gamma_j P'_{j-1},\ j=0,1,\ldots  \]
		and, by definition of $\eta_{i,j}$
		\[ P_j + x \sum_{i=0}^{j-1}\eta_{i,j}P_i = \alpha_j \sum_{i=0}^{j}\eta_{i,j+1}P_i + \beta_j \sum_{i=0}^{j-1}\eta_{i,j}P_i + \gamma_j \sum_{i=0}^{j-2}\eta_{i,j-1}P_i,\ j=0,1,\ldots  \]
		and so
		\begin{eqnarray*}
		\alpha_j \sum_{i=0}^{j}\eta_{i,j+1}P_i  & = &  P_j + \sum_{i=0}^{j-1}\eta_{i,j} (\alpha_i P_{i+1} + \beta_i P_{i} + \gamma_i P_{i-1}) \\
		& & \qquad\qquad\qquad - \beta_j \sum_{i=0}^{j-1}\eta_{i,j} P_i - \gamma_j \sum_{i=0}^{j-2}\eta_{i,j-1}P_i 
		\end{eqnarray*}
		rearranging indices and identifying similar coefficients, \eqref{eta} is achieved. That $\frac{d}{dx} {\cal P}={\cal P} N$ is a consequence of Proposition \ref{prop:L}.
	\end{proof} 
\end{proposition}

From that proposition we can derive explicit formulas for the first subdiagonals of $N$.

\begin{corollary}\label{propetaim1i}
	Under conditions of  { Proposition \ref{prop:dP}}, we have
	\begin{eqnarray}
	\eta_{j,j+1} & = & \frac{j+1}{\alpha_j},  \label{etasubdiag1}\\
	\eta_{j,j+2} & = & \frac{1}{\alpha_j\alpha_{j+1}} \sum_{i=0}^{j} ( \beta_i - \beta_{j+1} ), \label{etasubdiag2}\\
	\eta_{j,j+3} & = & \frac{1}{\alpha_j\alpha_{j+1}\alpha_{j+2}}\sum_{i=0}^{j} \left[ (\beta_i - \beta_{j+2})(\beta_i - \beta_{j+1}) + 2\alpha_i \gamma_{i+1} -\alpha_{j+1} \gamma_{j+2}\right],\label{etasubdiag3}
	\end{eqnarray}
	for $j\geq 0$.
	
	\begin{proof}
		For $j=0$, the result $\eta_{0,1}=\frac{1}{\alpha_0}$  is included in { Proposition \ref{prop:dP}}. Now, assuming that $\eta_{j-1,j}=\frac{j}{\alpha_{j-1}}$ and replacing in \eqref{eta} we recover \eqref{etasubdiag1} by induction.
		
		By definition of $\eta_{i,j}$ we can write 
		\[
		P'_2=\eta_{1,2}P_1+\eta_{0,2}
		\]
		and from \eqref{ttrr}
		\begin{eqnarray*}
			P'_2 & = & \frac{1}{\alpha_1}(P_1+x P'_1-\beta_{1}P'_1) \\
			& = & \frac{1}{\alpha_1}(P_1 + \frac{1}{\alpha_0}(\alpha_0 P_1 +\beta_0-\beta_{1}))
			= \frac{2}{\alpha_1}P_1+\frac{\beta_0-\beta_1}{\alpha_1\alpha_{0}}
		\end{eqnarray*}
		resulting in $\ \eta_{0,2} = \frac{\beta_0-\beta_1}{\alpha_0\alpha_{1}}\ $ and this confirms \eqref{etasubdiag2} for $j=0$.
		From \eqref{eta} we can write
		\[
		\eta_{j,j+2} = \frac{1}{\alpha_{j+1}}\left[ \alpha_{j-1}\eta_{j-1,j+1} + (\beta_{j}-\beta_{j+1}) \eta_{j,j+1} + \gamma_{j+1}\eta_{j+1,j+1} + \gamma_{j+1}\eta_{j,j} \right]
		\]
		and, from the fact that $\eta_{i,j}=0, i\geq j$ and from \eqref{etasubdiag1} we get 
		\[
		\eta_{j,j+2} = \frac{1}{\alpha_{j+1}}\left[ \alpha_{j-1}\eta_{j-1,j+1} + \frac{j+1}{\alpha_{j}} (\beta_{j}-\beta_{j+1}) \right]
		\]
		 and \eqref{etasubdiag2} follows by induction.
	\end{proof}
	
	The proof of \eqref{etasubdiag3} follows in a similar way, using \eqref{ttrr} with \eqref{etasubdiag1} and \eqref{etasubdiag2}.
\end{corollary}

A special case arises when in \eqref{ttrr} $\beta_i=0,\ i=0,1,\ldots$. This is the so-called symmetric case, resulting in the property that polynomials $P_n$ are functions with the same parity of $n$ and so, their derivatives $P'_n$ have the parity of $n-1$. Using {Proposition \ref{prop:dP}} we have an alternative proof to an equivalent result.

\begin{corollary}\label{prop:beta=0}
	Let $\cal P$ be a basis satisfying \eqref{ttrr} with 	$\beta_i=0,\ i=0,1,\ldots$ and $\eta_{i,j}$ defined by { Proposition \ref{prop:dP}}, then $\eta_{j-2k,j}=0,\ k=1,2,\ldots\lfloor j/2\rfloor,\  j\geq 2
	$,
where $\lfloor x\rfloor$ is the nearest integer less or equal to $x$.	
	\begin{proof}
		{Corollary \ref{propetaim1i}} with $\beta_i=0,\ i=0,1,\ldots,j$ proofs the result for $k=1$. Now, admitting that $\eta_{j-2,j}=\eta_{j-4,j}=\cdots =\eta_{j-2k-2,j}=0$ and taking $\beta_i=0$ in \eqref{ttrr}
		\begin{eqnarray*}
			\eta_{j-2k,j} & = & \frac{1}{\alpha_{j-1}}(\alpha_{j-2k-1}\eta_{j-2k-1,j-1}+\gamma_{j-2k+1} \eta_{j-2k+1,j-1}-\gamma_{j-1} \eta_{j-2k,j-2}) \\
			& = & \frac{\alpha_{j-2k-1}}{\alpha_{j-1}}\eta_{j-2k-1,j-1}
		\end{eqnarray*}
		by hypothesis. Iterating the last equality we arrive at
		\[
		\eta_{j-2k,j} = \frac{\alpha_{j-2k-1}\alpha_{j-2k-2}\cdots\alpha_{1}}{\alpha_{j-1}\alpha_{j-2}\cdots\alpha_{2k+1}}\eta_{0,2k}
		\]
		but
		\[
		\eta_{0,2k} = \frac{1}{\alpha_{2k-1}}(\gamma_{1} \eta_{1,2k-1}-\gamma_{2k-1} \eta_{0,2k-2})=0
		\]
		also by the same induction hypotheses.
	\end{proof} 
\end{corollary}

These results are also useful to derive the matrix representation of the primitive operator. 

\begin{proposition}\label{prop:IP}
	Let $\cal P$ be the basis satisfying \eqref{ttrr}. Defining  
	\[
	 O = [\theta_{ij}]_{i,j\geq 0},\quad \text{with}\quad \theta_{ij} = \frac{1}{||P_i||^2}<P_i, \int P_j >,
	\]
	then  for $j=1,2,\ldots$
	\begin{equation}\label{thetaij}
	\left\{\begin{array}{l}
	\theta_{j+1,j} = {\displaystyle \frac{\alpha_j}{j+1}} \\
	\theta_{i+1,j} = {\displaystyle -\frac{\alpha_i}{i+1} \sum_{k=i+2}^{j+1} \eta_{ik}\theta_{kj}}, i=j-1,\ldots,1,0
	\end{array}\right. 
	\end{equation}
	and  $\int {\cal P}={\cal P} O$.
	\begin{proof}
		By definition, considering that the primitive of $P_j$ is a polynomial of degree $j+1$ defined with an arbitrary constant term, we can write
		\[ \int P_j = \sum_{i=1}^{j+1}\theta_{ij}P_i,\ j=0,1,\ldots  \]
		Differentiating both sides and applying proposition \ref{prop:dP} we have
		\[
		P_j = \sum_{i=1}^{j+1}\theta_{ij}P'_i = \sum_{i=1}^{j+1}\theta_{ij}\sum_{k=0}^{i-1}\eta_{ki} P_k .
		\]
		Rearranging indices and identifying similar coefficients, 
		\[
		P_j = \sum_{i=0}^{j}\left[ \sum_{k=i+1}^{j+1} \eta_{ik} \theta_{kj}  \right] P_i .
		\]
		And so, for the coefficient of $P_j$,
		\[ 
		\eta_{j,j+1}\theta_{j+1,j} = 1
		\]
		and, for the coefficients of $P_i,\ i=0,\ldots,j-1$,
		\[
		\sum_{k=i+1}^{j+1}\eta_{ik}\theta_{kj} = 0
		\]
		The result is obtained solving for $\theta_{j+1,j}$ the first equation and for $\theta_{i+1,j}$ each one in the last set of equations.
		
	\end{proof} 
\end{proposition}

That proposition includes explicit formulas for the non null su-diagonal of matrix $O$. Using \eqref{thetaij} and { Corollary \ref{propetaim1i}} we get explicit formulas for $O$ matrix main diagonal and for the first upper diagonal.

\begin{corollary}\label{prop:thetajj}
	Under conditions of { Proposition \ref{prop:dP}}, we have for $j\geq 1$
	\begin{eqnarray}
	\theta_{j,j} & = & \frac{1}{j+1}(\beta_j- \overline{\beta}_j ),\\
	\theta_{j,j+1} & = & \frac{1}{(j+2)\alpha_{j}}( \frac{j+2}{j+1} \overline{\beta}_{j} \overline{\beta}_{j+2} + \alpha_{j}\gamma_{j+1} - \sigma_{j} - 2\xi_{j}) 
	\end{eqnarray}
	where $\overline{\beta}_j=\frac{1}{j}\sum_{i=0}^{j-1}\beta_i$, $\sigma_j=\frac{1}{j}\sum_{i=0}^{j-1}\beta_i^2$ and $\xi_j=\frac{1}{j}\sum_{i=0}^{j-1}\alpha_i\gamma_{i+1}$.
\end{corollary}

Next section is devoted to achieve explicit formulas to particular cases of classical orthogonal polynomials.

\subsection{Explicit formulas for Classical Orthogonal Polynomials} \label{subsec:COP}

In this section we treat the particular cases of the classical orthogonal polynomial basis, associated to the names of Jacobi, Laguerre, Hermite and Bessel. For the first three cases we follow \cite{AbSteg} handbook for definitions and normalizations. For Bessel polynomials data we follow \cite{KF}.

\subsubsection{Jacobi Polynomials}
	
Jacobi Polynomials $P_j^{(\alpha,\beta)}$ can be defined by \eqref{ttrr} with
\begin{equation*}    
	\alpha_j = \frac{2(j+1)(j+\gamma+1)}{(2j+\gamma+1)(2j+\gamma+2)},\
	\beta_j = \frac{\beta^2-\alpha^2}{(2j+\gamma)(2j+\gamma+2)},\
	\gamma_j = \frac{2(j+\alpha)(j+\beta)}{(2j+\gamma)(2j+\gamma+1)}
\end{equation*}                          
where $\alpha, \beta>-1$ are parameters and $\gamma=\alpha+\beta$. Those coefficients result in the following explicit formulas for the first two main diagonal elements in matrices $N$ and $O$.

\begin{proposition}
Let ${\cal P}=[P_i^{(\alpha,\beta)}]_{i\geq 0}$ be the Jacobi polynomials, $N = [\eta_{i,j}]_{i,j\geq 0}$ the differentiation matrix defined in {Proposition \ref{prop:dP}} and $O = [\theta_{i,j}]_{i,j\geq 0}$ the integration matrix defined in {Proposition \ref{prop:IP}}, then 
\begin{equation}\label{Jacobi_subdiags}
\begin{array}{rr}
\begin{cases}
\eta_{j,j+1}={\displaystyle \frac{(2j+\gamma+1)(2j+\gamma+2)}{2(j+\gamma+1)}} \\
\theta_{j+1,j}={\displaystyle \frac{1}{\eta_{j,j+1}}}
\end{cases} &
\begin{cases}
\eta_{j-1,j+1}={\displaystyle \frac{(2j+\gamma)^2 -1}{2(j+\gamma)(j+\gamma+1)}}(\beta-\alpha) \\
\theta_{j,j}={\displaystyle \frac{2(\alpha-\beta)}{(2j+\gamma)(2j+\gamma+2)}}
\end{cases}\\
 j=0,1,\ldots &  j=1,2,\ldots
\end{array}
\end{equation}
with $\gamma=\alpha+\beta$
\begin{proof}
The first set of equalities is obtained by direct substitution of $\alpha_j$ in \eqref{etasubdiag1} and in the first equality of \eqref{thetaij}. For the second set, from definition $\beta_j$ and by partial fraction decomposition we have
\begin{equation*}
\beta_j = (\beta-\alpha)\frac{\gamma}{(2j+\gamma)(2j+\gamma+2)} =  (\beta-\alpha)(\frac{j+1}{2j+\gamma+2} - \frac{j}{2j+\gamma})
\end{equation*}
and so
\begin{equation*}
	\sum_{i=0}^{j-1} \beta_i = (\beta-\alpha)\sum_{i=0}^{j-1} (\frac{i+1}{2i+\gamma+2} - \frac{i}{2i+\gamma}) =  (\beta-\alpha)\frac{j}{2j+\gamma}
\end{equation*}
then
\[
\sum_{i=0}^{j-1} \beta_i -j\beta_j = (\beta-\alpha)(\frac{j}{2j+\gamma} - \frac{j(j+1)}{2j+\gamma+2} + \frac{j^2}{2j+\gamma} ) =  \frac{2j(j+1)(\beta-\alpha)}{(2j+\gamma)(2j+\gamma+2)}
\]
and using {Corollary \ref{propetaim1i}}
\[ \eta_{j-1,j+1}={\displaystyle \frac{(2j + \gamma-1)(2j + \gamma+1)}{2(j + \gamma)(j + \gamma+1)}}(\beta - \alpha),\ j=1,2,\ldots \]
and from {Proposition \ref{prop:thetajj}}, $\theta_{j,j}=\frac{-1}{j(j+1)}[\sum_{i=0}^{j-1} \beta_i -j\beta_j]$.
\end{proof} 
\end{proposition}

From the Jacobi polynomials $J^{(\alpha,\beta)}$, a particular case arises when $\alpha=\beta$, the Gegenbauer polynomials.

\subsubsection{Gegenbauer Polynomials}

Gegenbauer Polynomials $C^{(\lambda)}$ are defined for parameter $\lambda>-\frac{1}{2},\ \lambda\neq 0$, by \eqref{ttrr} with
\begin{equation}\label{abgC}    
	\alpha_n = \frac{n+1}{2(n+\lambda)},\ 
	\beta_n = 0,\
	\gamma_n = \frac{n+2\lambda-1}{2(n+\lambda)}=1-\alpha_n
\end{equation}       
Those coefficients result in the following explicit formulas for the elements in matrices $N$ and $O$.

\begin{proposition}\label{prop:C}
	Let ${\cal P}=[C_i^{(\lambda)}]_{i\geq 0}$ be the Gegenbauer polynomials, $N = [\eta_{i,j}]_{i,j\geq 0}$ the differentiation matrix defined in {Proposition \ref{prop:dP}} and $O = [\theta_{i,j}]_{i,j\geq 0}$ the integration matrix defined in {Proposition \ref{prop:IP}}, then 
	\begin{equation}
	\begin{array}{rr}
	\begin{cases}
	\eta_{i,j}=(1-(-1)^{j-i})(\lambda+i),\ i<j \\
	\eta_{i,j}=0,\ i\geq j  
	\end{cases} &
	\begin{cases}
	\theta_{j\pm 1,j}={\displaystyle \frac{\pm 1}{2(\lambda+j)}} \\
	\theta_{i,j}=0,\ i\neq j\pm 1
	\end{cases}
	\end{array}, j\geq 0
	\end{equation}
	\begin{proof}
		Substituting \eqref{abgC} in \eqref{etasubdiag1} and in \eqref{etasubdiag2} results in $\eta_{j,j+1}=2(\lambda+j)$ and $\eta_{j-1,j+1}=0$ proving the result for $\eta_{j-1,j}$ and $\eta_{j-2,j}$. For $\eta_{j-2k,j},\ k=1,2,\ldots\lfloor j/2\rfloor,\  j\geq 1$ the result follows from {Proposition \eqref{prop:beta=0}}. Now the result is proved for $\eta_{i,j},\ i<j$ for columns $j\leq 2$. Admitting that this is true for columns $0,\ldots,j$ and introducing \eqref{abgC} in \eqref{ttrr} we have
		\begin{eqnarray*}
		\eta_{i,j+1} & = & \frac{\lambda+j}{j+1}\big[ \frac{i}{\lambda+i-1}\eta_{i-1,j} + \frac{2\lambda + i}{\lambda+i+1}\eta_{i+1,j} \big] -\frac{2\lambda+j-1}{j+1} \eta_{i,j-1} \\
		& = & \frac{\lambda+j}{j+1}\big[ \frac{i(1-(-1)^{j+1-i})}{\lambda+i-1}(\lambda+i-1) + \frac{(2\lambda + i)(1-(-1)^{j-1-i})}{\lambda+i+1}(\lambda+i+1) \big] \\
		& &  -\frac{2\lambda+j-1}{j+1} (1-(-1)^{j-1-i})(\lambda+i) \\
		& = & \frac{\lambda+j}{j+1}( 2\lambda + 2i) (1-(-1)^{j+1-i}) -\frac{2\lambda+j-1}{j+1} (1-(-1)^{j-1-i})(\lambda+i) \\
		& = & (1-(-1)^{j+1-i})(\lambda+i)
		\end{eqnarray*}    
	validating the formula for $\eta_{i,j+1}$ and, by induction over $j$, for the whole matrix $N$.
	  	
	For the elements of matrix $O$, from direct substitution of \eqref{abgC} in \eqref{thetaij} we get the formula to $\theta_{j+1,j}$ and
		\begin{eqnarray}\label{thetai+1j}
		\theta_{i+1,j} & = & \frac{-1}{2(\lambda+i)} \sum_{k=i+2}^{j+1} \eta_{i,k}\theta_{k,j} \nonumber \\
		& = & \frac{-1}{2(\lambda+i)} \sum_{k=i+2}^{j+1} (1-(-1)^{k-i})(\lambda+i)\theta_{k,j} = -\sum_{r=1}^{\lfloor \frac{j-i}{2}\rfloor} \theta_{i+1+2r,j} 
		\end{eqnarray}
	Since from {Proposition \ref{prop:thetajj}} we have $\theta_{j,j}=0$, we only have to iterate \eqref{thetai+1j} for $i=j-2,j-3,\ldots,1$. This results in
		\begin{eqnarray*}
		\theta_{j-1,j} & = & -\theta_{j+1,j} = -\frac{1}{2(\lambda+j)} \\
		\theta_{j-2,j} & = & -\theta_{j,j} = 0  \\
		\theta_{j-3,j} & = & -(\theta_{j-1,j} + \theta_{j+1,j}) = 0 \\
		\theta_{j-2k,j} & = & -\sum_{r=0}^{k-1} \theta_{j-2r,j} = 0
		\end{eqnarray*}
	\end{proof} 
\end{proposition}

One immediate consequence of that proposition is the validation of the exact formulas
\begin{eqnarray*} 
	\frac{d}{dx} C_j^{(\lambda)} & = & \sum_{i=0}^{j-1} (1-(-1)^{j-i})(\lambda+i)C_i^{(\lambda)} = 2\sum_{r=0}^{\lfloor \frac{j-1}{2}\rfloor} (\lambda+j-1-2r)C_{j-1-2r}^{(\lambda)} \\
	\int C_j^{(\lambda)} & = & \frac{1}{2(\lambda+j)}(C_{j+1}^{(\lambda)} -C_{j-1}^{(\lambda)})
\end{eqnarray*}

Next we present a set of formulas obtained to particular cases of sequences of classical orthogonal polynomials. The four Chebyshev cases result from Jacobi polynomials $P_n^{(\alpha,\beta)}$, combining $\alpha=\pm\frac{1}{2}$ with $\beta=\pm\frac{1}{2}$; the Legendre case results from Gegenbauer polynomials $C_n^{(\lambda)}$ with $\lambda=\frac{1}{2}$; formulas for the three other cases, Laguerre, Hermite and Bessel polynomials, result from the general orthogonal polynomials formulas, with the data provided. The proofs are particular cases of propositions {Proposition \ref{prop:dP}} to {Proposition \ref{prop:C}}.

\subsubsection{Particular Classical Orthogonal Polynomials}

In {Table \ref{tab:OPF}} we present the coefficients of the recurrence relation \eqref{ttrr} for a set of particular cases of classical orthogonal polynomials families. 

\begin{table}[hbp]
	\centering\large
	\caption{Values $\alpha_j,\ \beta_j,\ j=0,1,\ldots$ and $\gamma_j,\ j=1,2,\ldots$ for \eqref{ttrr}. $\delta_j\equiv \delta_{j,0}$ is the Kronecker symbol.}
\begin{tabular}{l|l|c|c|c}\hline\hline
	$P_j$ & Name & $\alpha_j$ & $\beta_j$ & $\gamma_j$ \\ \hline
	$C_j^{(\lambda)}$ & Gegenbauer & $\frac{j+1}{2(j+\lambda)}$& $0$ & $\frac{j+2\lambda-1}{2(j+\lambda)}$ \\
	$T_j$ & Chebyshev $1^{st}$ kind & $2^{\delta_{j}-1}$ & $0$ & $2^{-1}$ \\
	$U_j$ & Chebyshev $2^{nd}$ kind & $2^{-1}$ & $0$ & $2^{-1}$ \\
	$V_j$ & Chebyshev $3^{th}$ kind & $2^{-1}$ & $1-2^{-\delta_{j}}$ & $2^{-1}$ \\
	$W_j$ & Chebyshev $4^{th}$ kind & $2^{-1}$ & $2^{-\delta_{j}}-1$ & $2^{-1}$ \\
	$P_j$ & Legendre & $\frac{j+1}{2j+1}$ & $0$ & $\frac{j}{2j+1}$ \\
	$L_j$ & Laguerre & $-(j+1)$ & $2j+1$ & $-j$ \\
	$H_j$ & Hermite & $2^{-1}$ & $0$ & $2j$ \\
	$y_j$ & Bessel & $\frac{1}{2j+1}$ & $-\delta_j$ & $-\frac{1}{2j+1}$ \\ \hline
\end{tabular}
	\label{tab:OPF}
\end{table}

For the first and the second kinds of Chebyshev polynomials and for Legendre, Laguerre and Hermite polynomials, we follow the normalisation proposed by Hochstrasser in  \cite{AbSteg}. The data for Chebyshev polynomials of  third and fourth kinds   are from \cite{DLMF} and Bessel polynomials are defined by \cite{KF}. In some of those cases, coefficients $\alpha_0$ and $\beta_0$ are redefined to result the same cited polynomial sequences, with the initial conditions $P_{-1}=0,\ P_0=1$ as in \eqref{ttrr}.

In {Table \ref{tab:eta_theta}} and {Table \ref{tab:dP_IntP}} we present explicit formulas for the coefficients $\eta_{i,j}$ and $\theta_{i,j}$ of the Fourier expansions of $P'_j$ and $\int P_j$, respectively. $P_j$ are one of the orthogonal polynomials presented in {Table \ref{tab:OPF}}. Since $N=[\eta_{i,j}]$ are upper triangular matrices with null diagonal and $O=[\theta_{i,j}]$ are tridiagonal matrices, we present formulas only for the eventually non null elements. Any $\eta_{i,j}$ and $\theta_{i,j}$ elements with $i=-1$ or with $j=-1$ must be considered as zero.

\begin{table}[hbp]
	\centering\large
	\caption{$i,j\in\mathbb{N}$, $\oldstylenums{1}_{i,j}=(-1)^{i+j}$, $\oldstylenums{2}_{i,j} = 1-\oldstylenums{1}_{i,j}$, $\delta_i$ is the Kronecker symbol and $\varsigma_i=\delta_i-1$. In all cases, $P'_{0}=0$, $\int P_{0} = \alpha_{0}P_1$ and $\int P_{1} = \frac{\alpha_1}{2}P_2 + \frac{1}{2}(\beta_1-\beta_0)P_1$.}
\begin{tabular}{l|l|l|l}\hline\hline
	$P_j$ & $\eta_{i,j},\ i\geq 1,\ j>i$ & $\theta_{i,j},\ j\geq 2$ & $i$ \\ \hline
	$C_j^{(\lambda)}$ & $\oldstylenums{2}_{i,j}(i+\lambda)$ & $\pm (2j+2\lambda)^{-1}$ & $j\pm 1$\\
	$T_j$ & $\oldstylenums{2}_{i,j}j/2^{\delta_{i}}$ & $\pm (2i)^{-1}$ & $j\pm 1$ \\
	$U_j$ &  $\oldstylenums{2}_{i,j}(i+1)$  & $\pm (2j+2)^{-1}$ & $j\pm 1$ \\
	$V_j$ &  $\oldstylenums{2}_{i,j}(i+\frac{1}{2})+j-i$  & $\frac{1}{2}  (j+1)^{\varsigma_{j-1-i}} (-j)^{\varsigma_{j+1-i}}$ & $j,\ j\pm 1$\\
	$W_j$ & $(\oldstylenums{2}_{i,j}(i+\frac{1}{2})+j-i)\oldstylenums{1}_{i,j+1}$  & $-\frac{1}{2} (-j-1)^{\varsigma_{j-1-i}} j^{\varsigma_{j+1-i}}$ & $j,\ j\pm 1$\\
	$P_j$ & $\oldstylenums{2}_{i,j}(i+\frac{1}{2})$ & $\pm (2j+1)^{-1}$ & $j\pm 1$\\
	$L_j$ & $-1$ & $(-1)^{j-i}$ & $j,\ j+1$\\
	$H_j$ & $2j\delta_{j-1-i}$ & $(2j+2)^{-1}$ & $j+1$\\
	$y_j$ & $(i-j)(i+j+1)(i+\frac{1}{2})\oldstylenums{1}_{i,j}$ & $ (j+1)^{\varsigma_{j-1-i}} (2j+1)^{\varsigma_{j-i}} j^{\varsigma_{j+1-i}}$ & $j,\ j\pm 1$\\ \hline
\end{tabular}
\label{tab:eta_theta}
\end{table}

With those values, we get the polynomials derivatives and primitives expressed in the same orthogonal basis.

\begin{table}[htbp]
	\centering\normalsize
	\caption{Primitives are considered with undetermined $P_0$ coefficient. In all cases, $P'_{0}=0$, $\int P_{0} = \alpha_{0}P_1$ and $\int P_{1} = \frac{\alpha_1}{2}P_2 + \frac{1}{2}(\beta_1-\beta_0)P_1$.}
\begin{tabular}{l|l|l}\hline\hline
	$P_j$ & $P'_{j},\ j=1,2,\ldots$ & $\int P_{j},\ j=2,3,\ldots$ \\ \hline
	$C_j^{(\lambda)}$ & $\displaystyle 2\sum_{k=0}^{\lfloor (j-1)/2 \rfloor}(\lambda+j-1-2k)C_{j-1-2k}^{(\lambda)}$ & $\displaystyle \frac{1}{2(\lambda+j)}(C_{j+1}^{(\lambda)}-C_{j-1}^{(\lambda)})$ \\ \hline
	$T_j$ & $\displaystyle \frac{(1-(-1)^j)j}{2} T_0+2j\sum_{k=1}^{\lfloor  j/2 \rfloor}T_{j+1-2k}$ & $\displaystyle \frac{1}{2(j+1)}T_{j+1} - \frac{1}{2(j-1)}T_{j-1}$ \\ \hline
	$U_j$ & $\displaystyle 2\sum_{k=0}^{\lfloor (j-1)/2 \rfloor}(j-2k)U_{j-1-2k}$ & $\displaystyle \frac{1}{2(j+1)}(U_{j+1}-U_{j-1})$ \\ \hline
	$V_j$ & $\displaystyle \sum_{k=1}^{j}\big[ \frac{1-(-1)^k}{2}(2j + 1) + (-1)^k k\big]V_{j-k}$ & $\displaystyle \frac{V_{j+1}}{2(j+1)}-\frac{V_{j}}{2j(j+1)} -\frac{V_{j-1}}{2j}$ \\ \hline
	$W_j$ & $\displaystyle \sum_{k=1}^{j}\big[ \frac{1-(-1)^k}{2}(2j + 1) - k\big]W_{j-k}$ & $\displaystyle \frac{W_{j+1}}{2(j+1)}+\frac{W_{j}}{2j(j+1)} -\frac{W_{j-1}}{2j}$ \\ \hline
	$P_j$ & $\displaystyle \sum_{k=0}^{\lfloor (j-1)/2 \rfloor}(2j-1-4k)P_{j-1-2k}$ & $\displaystyle \frac{1}{2j+1}(P_{j+1}-P_{j-1})$ \\ \hline
	$L_j$ & $\displaystyle -\sum_{i=0}^{j-1}L_{i}$ & $\displaystyle L_j-L_{j+1}$ \\ \hline
	$H_j$ & $\displaystyle 2jH_{j-1}$ & $\displaystyle \frac{1}{2(j+1)}H_{j+1}$ \\ \hline
	$y_j$ & $\displaystyle \frac{1}{2}\sum_{i=0}^{j-1}(2i+1)(-1)^{j+i}(i-j)(i+j+1)y_{i}$ & $\displaystyle \frac{1}{2j+1} [\frac{y_{j+1}}{j+1}\! +\! \frac{y_{j-1}}{j}]\! +\! \frac{y_{j}}{j(j+1)}$ \\ \hline
\end{tabular}
\label{tab:dP_IntP}
\end{table}

Those matrices, introduced in propositions \ref{prop:xP}, \ref{prop:dP} and \ref{prop:IP}, can be interpreted as representing the action of multiplication by $x$, differentiating and integrating the elements of an orthogonal polynomial basis $\mathcal{P}$. In that sense, because they translate into algebraic terms those analytical operations, they are called Operational matrices. 

\section{Integro Differential operators}

Combined the operational matrices introduced in the previous section, we can translate in algebraic terms the action of a linear integro-differential operator over the coefficients of a formal Fourier  series.

\begin{proposition}
	If $y=\mathcal{P}a,\ a=[a_0,a_1,\ldots]^T$ is a formal Fourier series in the orthogonal polynomial basis $\mathcal{P}=(P_0,P_1,\ldots)$ satisfying  \eqref{PiPj} then
	\begin{enumerate}
		\item for $j\in\mathbb{N},\ x^j y = \mathcal{P}M^j a$;
		\item for $j\in\mathbb{N},\ \dfrac{d^j y}{dx^j} = \mathcal{P}N^j a$;
		\item for $p\in\mathbb{P},\ p(x) y = \mathcal{P}p(M) a$;
		\item for $j\in\mathbb{N}$ and $p\in\mathbb{P},\ p(x) \dfrac{d^j y}{dx^j} = \mathcal{P}p(M)N^j a$.
	\end{enumerate}
\end{proposition}
\begin{proof}
	Since $y=\mathcal{P}a$ then, using Proposition \ref{prop:xP}, $xy = x\mathcal{P}a = \mathcal{P} M a$. For $j>1$ the proof of sentence {\it 1.} results by induction over $j$.
	For sentence {\it 2.} the proof is the same, substituting the functional $l=x$ by $l=\dfrac{d}{dx}$ and using Proposition \ref{prop:dP}. In {\it 3.}, since $p\in\mathbb{P}$ we can write, for some $n\in\mathbb{N}$, $p=\sum_{i=0}^{n} p_i x^i$ and the result follows by linearity, and by {\it 1.} with $p(M)=\sum_{i=0}^{n} p_i M^i$ and {\it 4.} is a combination of {\it 2.} and {\it 3.}.
\end{proof}

Now, the action of a linear differential operator with polynomial coefficients over a formal Fourier series can be represented by an algebraic operation over the coefficients vector.

\begin{corollary}\label{prop:D}
	If $y=\mathcal{P}a,\ a=[a_0,a_1,\ldots]^T$ is a formal Fourier series in the orthogonal polynomial basis $\mathcal{P}=(P_0,P_1,\ldots)$ satisfying  \eqref{PiPj} and 
	\begin{equation}\label{Ddef}
	D = \sum_{i=0}^\nu p_i  \frac{d^i}{dx^i},\ p_i\in \mathbb{P}_{n_i} ,
	\end{equation}
	then $Dy = \mathcal{P}\Pi a$ where
	\begin{equation*}
	\Pi = \sum_{i=0}^\nu p_i(M)  N^i
	\end{equation*}
\end{corollary}

To extend this operational representation for integro-differential operators, we have to apply Proposition \ref{prop:IP} to the case of definite integrals. 

\begin{proposition}\label{prop:IaxP}
	Let $\cal P$ be the basis satisfying \eqref{ttrr}, $y=\mathcal{P}a$ a formal Fourier series and $O$ the matrix introduced in Proposition \ref{prop:IP}, defining  
	\[
	O_a^x = [\vartheta_{ij}]_{i,j\geq 0},\quad 
	\left\{\begin{array}{l}
	\vartheta_{0j} = {\displaystyle -\sum_{i=1}^{j+1}\theta_{ij}P_i(a) } \\
	\vartheta_{ij} = \theta_{ij},\ i>0
	\end{array}\right. 
	\]
	then $\int_{a}^{x}y = \mathcal{P} O_a^x a$
	\begin{proof}
		Defining $F_j\equiv \int P_j = \sum_{i=1}^{j+1}\theta_{ij}P_i$, the primitive with undefined $P_0$ coefficient, introduced in Proposition \ref{prop:IP}, then 
		\begin{eqnarray*}
			 \int_{a}^{x} P_j(t)dt & = & F_j(x) - F_j(a) \\
			 & = &  \sum_{i=1}^{j+1}\theta_{ij}P_i(x) - \sum_{i=1}^{j+1}\theta_{ij}P_i(a) = \mathcal{P}O_a^x e_j
		\end{eqnarray*}
		meaning that $\int_{a}^{x}\mathcal{P} = \mathcal{P} O_a^x $ in element wise sense. The proof follows by linearity. 
	\end{proof} 
\end{proposition}

In {Table \ref{tab:vartheta}} we present explicit formulas for the coefficients $\vartheta_{0j}$ of {Proposition \ref{prop:IaxP}} for the classical orthogonal polynomials defined in compact or in semi compact orthogonality intervals. In that table, we consider $\vartheta_{0j}$ as the coefficient of $P_0$ in $\int_a^x P_j(t)dt$ when integration limit $a$ is the same integration limit defining the orthogonality relation \eqref{PiPj}.

\begin{table}[hbp]
	\centering\large
	\caption{Coefficients $\vartheta_{0j}=\frac{1}{||P_0||^2}< P_0,\int_{a}^{x} P_j(t)dt >$.}
	\begin{tabular}{l|r|c|c|c}\hline\hline
		$P_j$ & $a$ & $\vartheta_{00}$ & $\vartheta_{01}$ & $\vartheta_{0j},\ j\geq 2$ \\ \hline
		$C_j^{(\lambda)}$ & $-1$ & $1$ & $-\frac{\lambda(2\lambda+1)}{2\lambda+2}$ & $\frac{(-1)^j}{j+1}\binom{j+2\lambda-2}{j}$\\
		$T_j$ & $-1$ & $1$ & $-1/4$ & $\frac{(-1)^{j+1}}{j^2-1}$ \\
		$U_j$ & $-1$ & $1$ & $-3/4$ & $\frac{(-1)^{j}}{j+1}$ \\
		$V_j$ & $-1$ & $1/2$ & $0$ & $\frac{(2j+1)(-1)^{j+1}}{j(j+1)}$ \\
		$W_j$ & $-1$ & $-1/2$ & $0$ & $\frac{(-1)^{j}}{j(j+1)}$ \\
		$P_j$ & $-1$ & $0$ & $1/6$ & $0$\\
		$L_j$ & $0$ & $1$ & $0$ & $0$\\ \hline
	\end{tabular}
	\label{tab:vartheta}
\end{table}

With repeated use of {Proposition} \ref{prop:IaxP} an equivalent result of {Proposition} \ref{prop:D} arrives to the case of integral operators.

\begin{proposition}
	If $y=\mathcal{P}a,\ a=[a_0,a_1,\ldots]^T$ is a formal Fourier series in the orthogonal polynomial basis $\mathcal{P}=(P_0,P_1,\ldots)$ satisfying  \eqref{PiPj} and 
	\begin{equation*}
	S = \sum_{i=1}^\nu p_i  I^i,\ p_i\in \mathbb{P}_{n_i},\quad I^iy = \int_{a_{ii}}^{x} \int_{a_{i,i-1}}^{x_{i-1}} \cdots  \int_{a_{i,1}}^{x_{1}} y(x_0) dx_0 dx_1 \cdots dx_{i-1}	
	\end{equation*}
	then $Sy = \mathcal{P}\varSigma a$ where
	\begin{equation*}
	\varSigma = \sum_{i=1}^\nu p_i(M) \varTheta_i,\quad \varTheta_i = O_{a_{ii}}^{x} O_{a_{i,i-1}}^{x} \cdots O_{a_{i,1}}^{x}
	\end{equation*}
\end{proposition}

\section{Numerical tests}

\subsection{Differential equations}

In order to test the robustness of both explicit and recurrence formulas presented in section \ref{Oper_on_Polyn}, we build matrices $M$ and $N$ associated to classical orthogonal polynomials and we test the effectiveness of their differential properties. Following \cite{AbSteg}, orthogonal polynomials satisfies differential equations of the type
\begin{equation} \label{DifEq}
	 g_2(x)P_j''+g_1(x)P_j'+a_j P_j=0,\ j=0,1,\ldots 
\end{equation}
where $g_1$ and $g_2$ are algebraic polynomials depending only on $x$ and $a_j$ are constants depending only on $j$. And so, for exact matrices $M$ and $N$, we must have
\begin{equation} \label{T} 
T\equiv g_2(M)N^2 + g_1(M)N + D = 0 
\end{equation}
where $D$ is the diagonal matrix $D=diag(a_0,a_1,\ldots)$ and $0$ the double infinite null matrix.
Table \ref{tab:DifEq} shows the data for equation (\ref{DifEq}), as in property 22.6 of \cite{AbSteg}. 

\begin{table}[hbp]
	\centering\large
	\caption{Coefficients $g_{1}$, $g_{2}$ and $a_n$ in (\ref{DifEq}).}
	\begin{tabular}{l|c|c|c}\hline\hline
		$P_j$ & $g_{2}(x)$ & $g_{1}(x)$ & $a_{n}$ \\ \hline
		$P_j^{(\alpha,\beta)}$ & $1-x^2$ & $\beta-\alpha-(\alpha+\beta+2)x$ & $n(n+\alpha+\beta+1)$  \\
		$C_j^{(\lambda)}$ & $1-x^2$ & $-(2\lambda+1)x$ & $n(n+2\lambda)$ \\
		$T_j$ & $1-x^2$ & $-x$ & $n^2$ \\
		$U_j$ & $1-x^2$ & $-3x$ & $n(n+2)$ \\
		$P_j$ & $1-x^2$ & $-2x$ & $n(n+1)$\\
		$L_j$ & $x$ & $1-x$ & $n$\\
		$H_j$ & $1$ & $-2x$ & $2n$\\ \hline
	\end{tabular}
	\label{tab:DifEq}
\end{table}

We test, for some families of orthogonal polynomials $P_j$ if the corresponding matrix $T_{n\times n}$ in (\ref{T}) is the exact null matrix evaluating $\max_{i,j}|T_{i j}|$, for some values of $n$. In Figure \ref{fig:Dif_Eqns_test} we present the results obtained with matrix dimensions starting with $n=20$ and stepping by $20$ to $n=1000$. We present values for Jacobi $P^{(\alpha,\beta)}$, Gegenbauer $C^{(\lambda)}$ and Legendre $P$ polynomials. For other cases, we have arrived in exact matrices.

\begin{figure}
	\includegraphics[width=\textwidth]{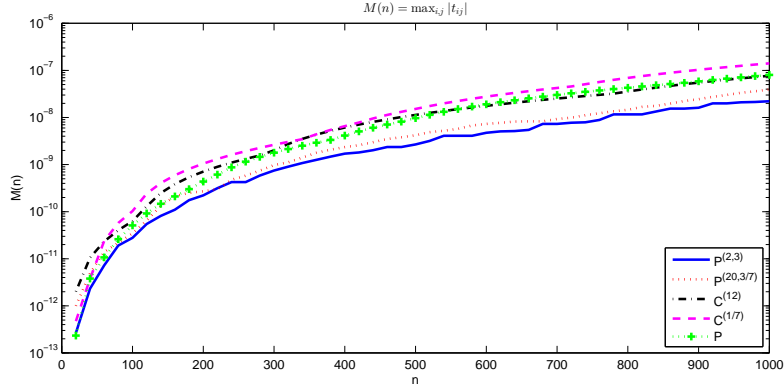}
	\caption{Error propagation $M(n)=\max_{ij}|t_{ij}|$ in $T=(t_{ij})_{n\times n}$, the matrix (\ref{T}) truncated to $n\times n$ dimension, for Jacobi $P^{(2,3)}$ and $P^{(20,3/7)}$, Gegenbauer $C^{(12)}$ and $C^{(1/7)}$ and for Legendre $P$ polinomials.}
	\label{fig:Dif_Eqns_test} 
\end{figure}

From those numerical experiences we can observe that for small values of $n$, the error propagation in the elements of matrices $T=g_2(M)N^2 + g_1(M)N + D$ evaluated in double precision arithmetic is absent or meaningless. For increasing values of $n$ we observe an increasing effect of the error propagation, in several cases of Jacobi polynomials, including Gegenbauer and Legendre particular cases. In those cases, the numerical behaviour of the recursive formulas (\ref{eta}) combined with explicit formulas (\ref{Jacobi_subdiags}), applied with Jacobi data have similar effect of using explicit data from {T}able \ref{tab:eta_theta}.

\subsection{Integral evaluation}

To test the formulas presented in {T}able \ref{tab:dP_IntP} we use properties 22.13 from \cite{AbSteg}. For Legendre polynomials case, and for each $n\in\mathbb{N}_0$ we must have
\[ I_{2n} = \int_{0}^{1} x^k P_{2n}(x)dx = \frac{(-1)^n\Gamma(\frac{1+k}{2})}{2\Gamma(-\frac{k}{2})} \frac{\Gamma(\frac{2n-k}{2})}{\Gamma(\frac{2n+3+k}{2})},\ k>-1 \] 
and
\[ I_{2n+1} = \int_{0}^{1} x^k P_{2n+1}(x)dx = \frac{(-1)^n\Gamma(\frac{2+k}{2})}{2\Gamma(\frac{2-k}{2})} \frac{\Gamma(\frac{2n+1-k}{2})}{\Gamma(\frac{2n+4+k}{2})},\ k>-2 \] 

Defining $P(x)=[P_0(x), P_1(x), \ldots,P_{n-1}(x)]$ we test, for several values of $k$ if the vector $T_n=(P(1)-P(0))M^k O$ retrieves the same exact values of $I_n=[I_0, I_1,\ldots,I_{n-1}]$, with truncated $n\times n$ matrices $M$ and $O$. We remark that, in Legendre case, we have exact values for $P_k(1)=1,\ k=0,1,\ldots,n$ and a recursive formula $P_{k+1}(0)=-\frac{k}{k+1}P_{k-1}(0),\ k=1,2,\ldots,n$ with $P_0(0)=1$ and $P_1(0)=0$ for values in the other extreme.

Figure \ref{fig:test_22_13_8}  show, for several valuresof $k$, the square norm of $\Delta_n=T_n-I_n$, obtained with double precision arithmetic, for increasing $n$, from $n=20$ up to $n=1000$.

\begin{figure}
	\includegraphics[width=\textwidth]{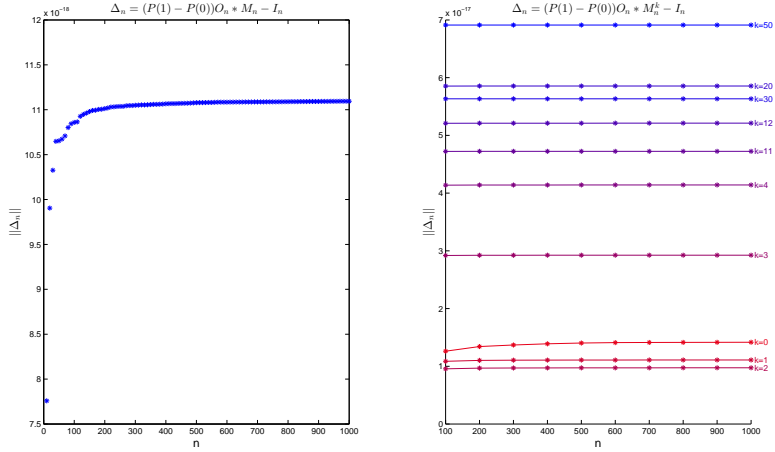}
	\caption{Error propagation $||(P(1)-P(0))M_{n}^k O_{n}-I_{n}||_2$  for distinct values for $k$ and $n=20,40,\ldots,1000$.}
	\label{fig:test_22_13_8} 
\end{figure}

We observe that the error propagation due to the evaluation of Legendre $M$ and $O$ operational matrices in double precision arithmetic, is slowly increasing with $n$ and with $k$. Despite this, even for $k=50$ and $n=1000$ the error square norm of $(P(1)-P(0))M_{n}^k O_{n}-I_{n}$ is negligible in double precision.

\subsection{Generating Functions}\label{GF}

To test the errors propagation in higher powers of $N$ and $O$ matrices, for some classical orthogonal polynomials $P$ we have selected an associated generating function, as in  \cite{AbSteg}, 
\[ g=\sum_{i=0}^{\infty} g_i z^i P_i(x).\] 
From those functions we build a set of differential operators $\mathcal{D}_k$ and integral operators $\mathcal{S}_k$ such that $\mathcal{D}_k(g)=0$ and $\mathcal{S}_k(g)=0$. So, if $D_n(k)$ and $S_n(k)$ are the matrices representing the action of $\mathcal{D}_k$ and $\mathcal{S}_k$, respectively, truncated to dimension $n\times n$, and if $a_n=(g_0, g_1 z, g_2 z^2,\ldots,g_{n-1}z^{n-1})$ are the coefficients vector of a partial sum of $g$, then $D_n(k)a_n$ and $S_n(k)a_n$ are residuals vectors, approaching the null vector as $n$ goes to infinity.

In {T}able \ref{tab:test_22_9} we present matrices $D(k)$ and $S(k)$, together with the associated generating functions, selected for our numerical tests.

\begin{table}[hbp]
	\centering\large
	\caption{Matrices $D(k)$, $S(k)$, coefficients $g_i$ and generating function $g$ in (\ref{GF}), $R=(1+z^2)I-2zM$, $r=1+z^2-2zx$.}
	\begin{tabular}{l|l|l|l|l}\hline\hline
		$P_j$ & $D(k)$ & $S(k)$ & $g_{i}$ & $g$ \\ \hline
		$C_j^{(\lambda)}$ & $R N^{k+1}-2(\lambda+k)z N^{k}$ & $(2z)^k (\lambda-k)_k O^k-R^k$  & $1$ & $r^{-\lambda}$ \\
		$T_j$ & $R N^{k+1}-2kz N^{k}$ &   & $\frac{1}{i}$ & $1-\frac{1}{2}\ln(r)$ \\
		$U_j$ & $R N^{k+1}-2(k+1)z N^{k}$ &   & $1$ & $r^{-1}$ \\
		$P_j$ & $R N^{k+1}-(2k+1)z N^{k}$ & $(2k-1)!!z^k O^k-(-R)^k$  & $1$ & $r^{-1/2}$ \\
		$L_j$ & $(z-1)^k N^{k} - z^k I $ & $z^k O^{k} - (z-1)^k I $  & $1$ & $\frac{e^{zx/(z-1)}}{1-z}$ \\
		$H_j$ & $N^{k} - (2z)^k I $ & $(2z)^k O^{k} - I $  & $\frac{1}{i!}$ & $e^{2zx-z^2}$ \\ \hline
	\end{tabular}
	\label{tab:test_22_9}
\end{table}

\begin{figure}
	\includegraphics[width=\textwidth]{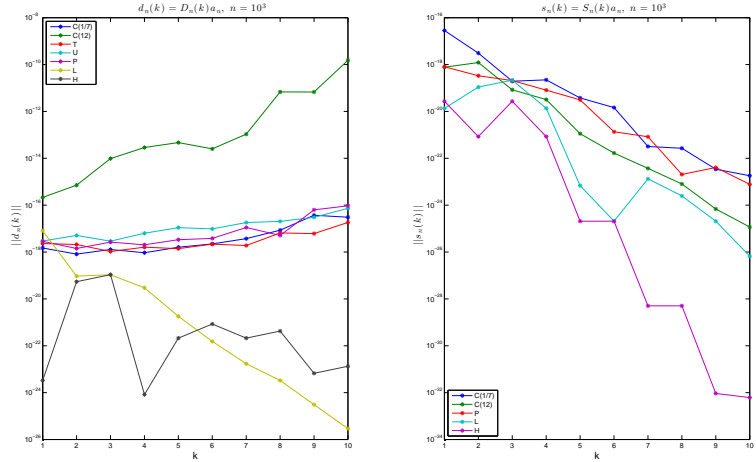}
	\caption{Error propagation $||d_n(k)||_2$, and $||s_n(k)||_2$ and $a_n$ results from $G(k)$ and $a$, as defined in {T}able \ref{tab:test_22_9}, truncated to dimension $n=1000$, $z=1/10$.}
	\label{fig:test_22_9} 
\end{figure}

\section{Conclusions} 
In this paper we have introduced a set of formulas intended to evaluate the generalized Fourier coefficients of the transforms of orthogonal polynomials by integro-differential operators. These formulas are suitable to evaluate the whole set of coefficients of an orthogonal polynomials base transformed by integral and differential operations.

In the case of general orthogonal polynomial bases, we have recursive formulas, allowing to evaluate the matrices representing the action of integral and differential operators over those bases. For the case of the most used classical orthogonal polynomials, from the recursive general formulas, we arrive at explicit formulas. Some of those formulas are already known in the literature, but we believe some others  are quite new.

Another feature of these formulas is that all of them result only from the recurrence relation characteristic of orthogonal polynomials, avoiding the need for additional information. This allows to obtain   the matrix representation of the action of integral and differential operations, represented in an orthogonal polynomial basis, only from the parameters of its three terms recurrence relation.

Finally, we have proposed numerical tests to stress the behaviour of those matrix representations when evaluated in finite arithmetic. Based on orthogonal polynomials properties, we measured the error propagation introduced by double-precision arithmetic, in several integral and differential operations. Numerical results indicate that the formulas introduced in this work are sufficiently robust to produce highly accurate results, even when dealing with powers of high dimensional matrices.


\begin{thebibliography}{spbasic}
%
\bibitem{AbSteg}
M. Abramowitz and I. Stegun, Handbook of Mathematical Functions. Dover Publications, New York, 9th ed. (1972)


\bibitem{CW}
R.Y. Chang, M.L. Wang, Shifted Legendre Direct Method for Variational Problems, Journal of Optimization theory and Applications, 39(2), 299--307 (1983)

\bibitem{DLMF}
NIST Digital Library of Mathematical Functions. http://dlmf.nist.gov/, Release 1.0.13 of 2016-09-16. F. W. J. Olver, A. B. Olde Daalhuis, D. W. Lozier, B. I. Schneider, R. F. Boisvert, C. W. Clark, B. R. Miller, and B. V. Saunders, eds.

\bibitem{Doha}
E.H. Doha, Explicit Formulae for the Coefficients of Integrated Expansions of Jacobi Polynomials and Their Integrals, Integral Transforms and Special Functions, 14(1), 69--86 (2003)

\bibitem{DA}
E.H. Doha and H.M. Ahmed, On the coefficients of integrated expansions of Bessel  polynomials, J. Comput. Appl. Math., 187, 58--71 (2006)


\bibitem{DSW}
J.J. Dongarra, B. Straughan and D.W. Walker, Chebyshev Tau-QZ algorithm for calculating spectra of hydrodynamic stability problems, Applied Numerical Mathematics, 22, 399--434 (1996)



\bibitem{Greengard}
L.Greengard, Spectral integration and two-point boundary value problems, SIAM J. Numer. Anal., 28, 1071--1080 (1991)


\bibitem{Koepf}
W. Koepf, Identities for families of orthogonal polynomials and special functions, Integral Transforms and Special Functions, 1-2, 66--102 (1997)

\bibitem{KD}
W. Koepf, D. Schmersau, Representations of orthogonal polynomials, J. Comput. Appl. Math., 90, 57--94 (1998)

\bibitem{KW}
W. Kong and X. Wu, Chebyshev Tau matrix method for Poisson-type equations in irregular domain, J. Comput. Appl. Math., 228, 158--167 (2001)

\bibitem{KF}
Hi L. Krall, Orrin Frink, A new class of orthogonal polynomials: The Bessel polynomials, Transactions of the American Mathematical Society, 65(1), 100--115 (1949)

\bibitem{MMR} 
J. C. de Matos, J. Matos, M.J. Rodrigues, Filtering the Tau method with Frobenius-Pad\'e approximants, submitted.

\bibitem{OS}  E.L. Ortiz, H. Samara, An Operational Approach to the
Tau Method for the Numerical solution of Nonlinear Differential Equations, 
Computing, 27(4), 15--21 (1981)

\bibitem{Phillips}
T.N. Phillips, On the Legendre coefficients of a general-order derivative of an infinitely differentiable function, IMA J. Numer. Anal. 8, 455-459 (1988)


\bibitem{SKR}
M. Shaban, S. Kazem and J.A. Rad, A modification of the homotopy analysis method based on Chebyshev operational matrices, Mathematical and Computer Modelling, 57, 1227--1239 (2013)

\bibitem{TVM17}
M. S. Trindade, P. B. Vasconcelos and J. Matos, Dealing with non-polynomial coefficients within tau method, Mathematics in Computer Science, submitted

\bibitem{VMT17} 
P. B. Vasconcelos, J. Matos and M. S. Trindade, Spectral Lanczos'tau method for systems of nonlinear integro-differential equations In: Constanda C., Dalla Riva M., Lamberti P., Musolino P. (eds) Integral Methods in Science and Engineering, Volume 1. Birkhäuser, Cham, (2017)


\end{thebibliography}


\end{document}